\documentclass{amsart}

\usepackage{amssymb}

\newtheorem{theorem}{Theorem}[section]
\newtheorem{lemma}[theorem]{Lemma}
\newtheorem{proposition}[theorem]{Proposition}
\newtheorem{corollary}[theorem]{Corollary}
\theoremstyle{definition}

\theoremstyle{remark}
\newtheorem{remark}[theorem]{Remark}
\newtheorem*{thmi}{Theorem I}
\numberwithin{equation}{section}

\newcommand{\R}{{\mathbb R}}

\renewcommand{\H}{\mathcal{H}}

\author{Mark Allen}
\address{Department of Mathematics, Purdue University, West Lafayette,
  IN 47907}
\email{allenma@math.purdue.edu}
\thanks{M.~Allen is supported by  Purdue Research Foundation and in part by NSF grant DMS-1101139}

\title[Free Boundary Problem]{Separation of a lower dimensional free boundary in a two phase problem}

\begin{document}

\begin{abstract}
We study minimizers of the energy functional
\[   
\int_{D}{|x_n|^a |\nabla u|^2} + \int_{D \cap (\R^{n-1} \times \{0\} )}{\lambda^+ \chi_{ \{u > 0\} } + \lambda^- \chi_{ \{u<0\} }} \ d\H^{n-1}
\]
without any sign restriction on the function $u$. The main result states that the two free boundaries 
\[
\Gamma^+ = \partial \{u(\ \cdot \ , 0) > 0\}  \text{ and }
\Gamma^- = \partial \{u(\ \cdot \ , 0) < 0\} 
\]
cannot touch. i.e. $\Gamma^+ \cap \Gamma^- = \emptyset$
\end{abstract}

\maketitle
\section{Introduction}    \label{S: introduction}
This paper aims to study the local properties of a two phase free boundary problem for the fractional Laplacian.  
Recently, in \cite{mA11} the following free boundary problem for the half laplacian has been studied.
For a function $u \in C(\R^N)$ and domain $D$ consider the problem
\begin{equation}\label{E: fractional1}
\begin{aligned}
(-\Delta)^{1/2} u(x)=0 &\quad\text{in }D \cap \{u > 0\}\\
\lim_{y \to x} \frac{u(y)}{((y-x) \cdot \nu(x))^{1/2}} = A &\quad\text{ if } x \in D \cap \partial\{u=0\}
\end{aligned}
\end{equation}
The study of \eqref{E: fractional1} presents certain difficulties since the fractional laplacian is a global operator. Many of the common techniques for studying free boundaries are unavailable. However, one may work in $\R^n = \R^{N+1}$ and a common reformulation of the half laplacian is the following
\[
(- \Delta)^{1/2}u(x') = \lim_{x_n \to 0} \tilde{u}_{x_n}(x',x_n) \text{  for } x' \in \R^N
\] 
where 
\[
\begin{aligned}
\Delta \tilde{u} =0 \text { in } \R_{+}^{N+1} \\
\tilde{u}(x',0) = u(x')
\end{aligned}
\]
By adding the extra dimension one may then study a localized version of the free boundary problem \eqref{E: fractional1} by studying minimizers of the functional
\[
\int_{D}{|\nabla u|^2} + \int_{D \cap (\R^{n-1} \times \{0\} )}{ \chi_{ \{u > 0\} } } \ d\H^{n-1}
\]
Since the above functional gives study to a one phase problem, it is natural to study the corresponding two phase problem
\[
\int_{D}{|\nabla u|^2} + \int_{D \cap (\R^{n-1} \times \{0\} )}{\lambda^+ \chi_{ \{u > 0\} } + \lambda^- \chi_{ \{u<0\} }} \ d\H^{n-1}
\]
which has been done in \cite{mA11}. One may generalize the study of the free boundary problem \eqref{E: fractional1} by considering other powers $0<s<1$ of the fractional laplacian. In \cite{CS} the appropriate extension theorem was proven enabling one to give a reformulation of the fractional laplacian by adding an extra dimension. By adding an extra dimension, the study of \eqref{E: fractional1} with $s$ replacing $1/2$ is reduced to the study of the minimizers of the functional
\[
\int_{D}{|x_n|^a|\nabla u|^2} + \int_{D \cap (\R^{n-1} \times \{0\} )}{\chi_{ \{u > 0\} }} \ d\H^{n-1}
\]
where $a=1-2s$ and $n=N+1$. 
This problem has been recently studied in \cite{lC10}. This paper will study the corresponding two-phase problem and extend one of the main results in \cite{mA11} to the general fractional case when $0<s<1$. 
This paper will then focus on the localized two phase problem which is to consider minimizers of the functional
\begin{equation}  \label{E: functional}
\int_{D}{|x_n|^a|\nabla u|^2} + \int_{D \cap (\R^{n-1} \times \{0\} )}{\lambda^+ \chi_{ \{u > 0\} } + \lambda^- \chi_{ \{u<0\} }} \ d\H^{n-1}
\end{equation}
over the class
\[
H^1(a,D) \overset{\text{def}}{=} \{v \in L^2(D) \mid |y|^{a/2} \nabla v \in L^2(D)\}
\]
and such that $u - \phi \in H_0^1(a,D)$ for a prescribed $\phi$. From the relation $0<s<1$ and $s=(1-a)/2$ it follows that $a$ will vary in the range $-1<a<1$. Throughout the paper we assume that $\lambda^+$ and $\lambda^-$ are positive constants. By use of the extension theorem given in \cite{CS} it is natural to make the assumptions that  $D$ and $\phi$ are symmetric about the hyperplane $\R^{n-1} \times \{0\}$; however, we will not make these assumptions in this paper since the proofs presented will not rely on even symmetry.  

The main study of this paper concerns the local properties of the two free boundaries
\[
\Gamma^+ = \partial \{u( \ \cdot \ , 0) > 0 \} \text{ and } \Gamma^- = \partial \{u( \ \cdot \ , 0) < 0 \}
\]
with the boundary being defined by the topology of $\R^{n-1} \times \{0\}$.
 
\subsection*{Motivation and Applications}
The motivation for studying the problem \eqref{E: functional} comes from recognizing the similarity to the problem of studying minimizers of the functional
\begin{equation}    \label{E: classical}
J(u) = \int_{D}{|\nabla u|^2 + \lambda^+ \chi_{ \{u>0\} }  + \lambda^- \chi_{ \{u<0\} }}
\end{equation}
which has been done in \cite{MR732100}. 
The minimizers of \eqref{E: classical} are generalized solutions of a classical two-phase free
boundary problem
\begin{equation}\label{E: FBP-classic}
\begin{aligned}
\Delta u=0&\quad\text{in }\{u>0\}\cup\{u<0\}\\
|\nabla u^+|^2-|\nabla u^-|^2=M&\quad\text{on }\partial\{u>0\}\cup\partial\{u<0\},
\end{aligned}
\end{equation}
with $M=(\lambda^+)^2-(\lambda^-)^2$. The study of problem \eqref{E: FBP-classic} has applications in two dimensional flow problems as well as in heat flow. In one specific application, the problem \eqref{E: FBP-classic} arises in a simplified model for premixed equidiffusional flames, in the stationary case, in the limit as $\epsilon \to 0+$ of a singular perturbation problem, see e.g.\ \cite{CLW}. By measuring the positivity and negativity on the boundary $\R^{n-1} \times \{0\}$, minimizers of \eqref{E: functional} can be seen as the limit of solutions to a boundary reaction problem, see e.g. \ \cite{mA11}.

In modeling, when long range interactions are present, it is relevant to replace
the Laplacian by nonlocal operators, such as the fractional
Laplacian. See survey papers \cite{MR1937584} and 
\cite{MR1081295}. 

\subsection*{Main Results}
As previously mentioned the two phase case of \eqref{E: functional} has been recently studied in \cite{mA11} under the additional assumption that $a=0$ (or $s=1/2$). By restricting the two phase problem to the case in which $a=0$, the authors in \cite{mA11} were able to use more technical tools such as the Alt-Caffarelli-Friedman monotonicity formula. One of the main results in \cite{mA11} is that if $a=0$, then $\Gamma^+ \cap \Gamma^- = \emptyset$. i.e. the free boundaries $\Gamma^+$ and $\Gamma^-$ cannot touch. This result is in complete contrast to many two phase free boundary problems. Often the interphase  $\Gamma^+ \cap \Gamma^-$ is difficult to study. In the classical two-phase free boundary problem in \eqref{E: FBP-classic}, the two-phase points create a major
complication even in the proof of the optimal (Lipschitz in that case)
regularity of solutions, see \cite{MR732100}. The separation of $\Gamma^+$ and $\Gamma^-$ is useful in that it reduces the two-phase free boundary problem to the one-phase free boundary problem. That is, locally minimizers have a sign in $\R^{n-1} \times \{0\}$, and so we may assume either $\lambda^+=0$ or $\lambda^- =0$.
In the case $a=0$, after one establishes optimal regularity and nondegeneracy, the separation of $\Gamma^+$ and $\Gamma^-$ is an immediate consequence of the Alt-Caffarelli-Friedman (ACF) monotonicity formula which was introduced and proven in \cite{MR732100}. The main result of this paper is the separation of the free boundaries for the more general case in which $a \neq 0$, namely

\begin{thmi}   \label{T: acase}
Let $-1<a<1$ and let $u$ be a minimizer to the functional in \eqref{E: functional}. Then $\Gamma^+ \cap \Gamma^- = \emptyset$. Furthermore, if $x_0 \in \Gamma^+ \ (x_0 \in \Gamma^-)$ then there exists $r>0$ such that $u \geq 0 \ (u \leq 0)$ in $B_r(x_0)$.
\end{thmi}

The ACF monotonicity formula provides a simple proof to Theorem I when $a=0$. Therefore, it would be natural in seeking to prove Theorem I to try to prove a generalization of the ACF monotonicity formula that applies to solutions of div$(|x_n|^a \nabla u) \geq 0$. Unfortunately, the proof of such a formula would require much more than mere adaptations to the proof of the classical ACF monotonicity formula.  In this paper we provide an alternate method that gives a relatively simple proof of Theorem I and does not utilize a generalization of the ACF formula. In its place we utilize a Weiss-type  monotonicity formula (defined in Section \ref{S: nondegeneracy}) that is an adaptation of the Weiss-type monotonicity formula given in \cite{mA11}.  The proof that the functional in \eqref{E: weiss} is monotone requires only slight modifications of the proof provided in \cite{mA11}. 

\subsection*{Outline of Paper}
The outline of this paper is as follows.

- In Section \ref{S: weights} we state known results for the weight $|x_n|^a dx$ and solutions of div$(|x_n|^a  \nabla u) =0$ that we will need. 

- In section \ref{S: optimalregularity} we prove the optimal regularity of minimizers and its corollaries. 

- In section \ref{S: nondegeneracy} we use nondegeneracy and the Weiss monotonicity formula to prove that ``blow-ups'' (see \eqref{E: rescale}) of minimizers are homogeneous of degree $s=(1-a)/2$ 

- In section \ref{S: courant} we use the Courant-Fischer maximum-minimum principle to establish a lower bound for the degree of homogeneity for homogeneous solutions of div$(|x_n|^a \nabla u)=0$. 

- In section \ref{S: separation} we use the results from the previous sections to provide a simple proof of Theorem I. 

\subsection*{Notation and Terminology}
For the remainder of the paper it will be useful to use the following notation.
$B_r(x_0) := \{x \in \R^n \mid \ |x-x_0| \leq r \}$ and $B_r = B_r(0)$ the ball centered at the origin with radius r. 
We denote a point $x \in \R^n$ by $(x',x_n)$ where $x' = (x_1, \ldots, x_{n-1})$.

For any set $\Omega \subset \R^n$, we define
\[
\Omega' \overset{ \text{def} }{=} \Omega \cap (\R^{n-1} \times \{0\} )
\]
Throughout the paper we will refer to the plane $\R^{n-1} \times \{0\}$ as the thin space. Likeweise, we will call $B_{r}'$ the thin ball where as $B_r$ will be the solid ball. 
For minimizers of \eqref{E: functional} we will call the set 
\[
\Lambda(u) = (\R^{n-1} \times \{0\}) \cap \{u=0\}
\]
the coincidence set. 

We define the following two spaces
\[
\begin{aligned}
H^1(a,D) & \overset{\text{def}}{=} \{v \in L^2(D) \mid |x_n|^{a/2} \nabla v \in L^2(D)\} \\
L^2(a,D) & \overset{\text{def}}{=} \{|x_n|^{a/2} v \in L^2(D)\}
\end{aligned}
\]
We will also use $\mathcal{L}_a u$ to denote the operator div$(|x_n|^a \nabla u)$. Throughout the paper $s=(1-a)/2$. 

\section{p-admissible weights and a-harmonic functions}   \label{S: weights}
We begin this section by noting that the measure $|x_n|^a dx$ is a Muckenhoupt $A_2$ weight. In \cite{MR1207810} it is shown that Muckenhoupt $A_p$ weights are $p$-admissible weights; therefore, we have the following Sobolev inequality from \cite{MR1207810}
\begin{equation}    \label{E: asobolev}
\left( \frac{1}{| B |_{a}}  \int_{B}{|\phi|^{2 \varkappa} |x_n|^a \ dx}  \right)^{\frac{1}{2 \varkappa}}  \leq cr 
\left( \frac{1}{| B |_{a}}  \int_{B}{|\nabla \phi|^{2} |x_n|^a \ dx } \right)^{\frac{1}{2}}
\end{equation}
whenever $B=B(x_0,r)$ is a ball and $\phi \in H_0^1(a,B)$. $\varkappa >1$ and $c$ are two constants depending on $n$ and $a$. Here, $| B |_a = \int_{B}{|x_n|^a dx}$. 

The following proposition is a consequence of the compactness theorem for admissable p-weights proven in \cite{MR1455468}.

\begin{proposition}[Rellich-Kondrachov Compactness]   \label{P: compact}
Let $u_k$ be a bounded sequence in $H_0^1(a, D)$ for $D \Subset \R^n$. Then there exists a convergent subsequence such that $u_k \to u$ pointwise $a.e.$ and in norm in $L^q(a,D)$ for all $q<2 \varkappa$ for $\varkappa$ as in \eqref{E: asobolev} .
\end{proposition}

We call a function $u$ $a$-harmonic if $\mathcal{L}_a u =0$. These functions share many properties with classical harmonic functions. In \cite{MR643158} it is shown that $a$-harmonic functions are H\"older continuous. It was also shown that $a$-harmonic functions have the maximum principle, Harnack inequality, and Boundary Harnack inequality. We also have the following Almgren's type monotonicity formula which was proven in \cite{CS}.
\begin{lemma}    \label{L: almgren}
Let $\mathcal{L}_a u =0$ in $B_1$. Then 
\[
N(r,u)=r\frac{\int_{B_r}{|x_n|^a |\nabla u|^2}}{\int_{\partial B_r}{|x_n|^a u^2}} = r \frac{D(r)}{H(r)}
\]
is monotone increasing in $r$. $N(r,u)$ is constant if and only if $u$ is homogeneous of degree $k$. 
\end{lemma}

Our assumptions are slightly different from those given in \cite{CS}; namely we do not assume even symmetry in the $x_n$ variable. The modified proof is therefore placed in the appendix.

\begin{lemma}   \label{L: lowerbound}
Let $\mathcal{L}_a u =0$ in $B_1$ with $u$ not identically zero. Assume also that $u(0)=0$. Then 
\[
\lim_{r \to 0} N(r,u) = N(0+,u) \geq \min \{1, 1-a\}
\] 
\end{lemma}

\begin{proof}
It is easy to verify that $N(\rho, u_r)=N(r\rho, u)$ for the rescalings 
\[
u_r(x) := \frac{u(rx)}{\frac{1}{r^{n-1+a}}\int_{\partial B_r}{|x_n|^a u^2}}
\]
and $\| u_r \|_{L^2(a,\partial B_1)} =1$. Now $\mathcal{L}_a u_r =0$ in $B_{1/r}$ and from the uniform H\"older continuity  provided in \cite{MR643158} and the Sobolev inequality \eqref{E: asobolev}, we may extract a subsequence such that $u_r \to u_0$ in $C^{\beta}(B_{\rho})$ and weakly in $H^1(a, B_{\rho})$ for $\rho <1$. The strong convergence in $H^1(a, B_{\rho})$ follows by using the Caccioppoli inequality for $a$-harmonic functions
\[
\int_{B_{\rho}}{|\nabla(u_r - u_0)|^2 |x_n|^a} \leq \frac{C}{(r-\rho)^2}\int_{B_r}{|u_r-u_0|^2 |x_n|^a}
\]
Now 
\[
N(\rho, u_0) = \lim_{r \to 0} N(\rho, u_r) = \lim_{r \to 0} N(r \rho, u) = N(0+,u) 
\]
So $\mathcal{L}_a u_0 =0$ in $B_1$ and is homogenous of degree $k = N(0+,u)$. $u_0$ is not identically zero since $\| u_0 \|_{L^2(a,\partial B_1)} =1$. We now only need to conclude that $k \geq \min \{1, 1-a\}$. Since $u_0(0)=0$, this is a direct consequence of Theorem \ref{T: originregularity}.  Theorem \ref{T: originregularity} has been placed in Section \ref{S: courant} for purposes of readibility of the paper. 
\end{proof}

From Almgren's monotonicity formula we may prove the following Lemma.

\begin{lemma}  \label{L: menergy}
If $\mathcal{L}_a u =0$ in $B_R(y',0)$  then 
\[
\frac{1}{r^{n-|a|}} \int_{B_r(y',0)}{|x_n|^a |\nabla u|^2}
\]
is monotone increasing in $r$.
\end{lemma}

A few remarks need to be said. First, if we add the additional assumption for even symmetry, namely that that $u(x',x_n) = u(x', -x_n)$, then 
\begin{equation}   \label{E: na}
\frac{1}{r^{n+a}}\int_{B_r(y',0)}{|x_n|^a  |\nabla u|^2} \quad  \text{ is monotone increasing in } r
\end{equation}
The solution $v=\frac{x_n}{|x_n|^a}$ for $a>0$ shows that if there is not even symmetry, then \eqref{E: na} is not true. Likewise, the hypothesis that that the ball be centered on the $\R^{n-1} \times \{0\}$ plane is essential. $v$ as given above with center $(y',y_n)$ with suitably chosen $y_n \neq 0$ will be a counterexample.   \eqref{E: na} is also not true if the ball is not centered on the thin space, and a counterexample is much easier to provide: off the thin space solutions are $C^1$, so if $a>0$ and $y_n \neq 0$ then
\[
\lim_{r \to 0} \frac{1}{r^{n+a}} \int_{B_r(y',y_n)}{|x_n|^a |\nabla u|^2} \to \infty
\] 
and so it is clear that \eqref{E: na} can only be true if $y_n=0$.   

\begin{proof}
Following the notation in Lemma \ref{L: almgren} we have 
\[
H'(r)= \frac{(n-1+a)}{r} H(r) + 2D(r)
\]
This equality comes from \eqref{E: almgrenlemma2} and \eqref{E: der}. This implies that $rH'(r)/H(r) = n-1+a+2N(r)$ is also monotone increasing. Hence $rH'(r)/H(r) \geq n-1+a+2k$ where $k=N(0+)$. Then $r^{-(n-1+a+2k)}H(r)$ is monotone increasing and therefore also
\[
\frac{1}{r^{n-2+a+2k}}D(r) = \frac{1}{r^{n-1+a+2k}}H(r)N(r) 
\]
is monotone increasing in $r$. By subtracting the constant $u(0)$ which is a solution of $\mathcal{L}_a$ we may use Lemma \ref{L: lowerbound} to conclude that $k \geq \min \{1,1-a\}$ and the Lemma is proven.
\end{proof}

\section{Optimal Regularity}    \label{S: optimalregularity}
In studying free boundary problems it becomes useful to utilize the so called ``blow-up'' process. If $u$ is a minimizer of the functional \eqref{E: functional} in $B_1(x_{0}',0)$, then the rescaled function
\begin{equation}   \label{E: rescale}
u_r(x) \overset{\text{def}}{=} \frac{u((x_{0}',0)+rx)}{r^{s}}
\end{equation}
is a minimizer in $B_{1/r}$. Here $s=(1-a)/2$. By taking a sequence $r_k \to 0$ we may hope to find a subsequence $u_{r_k} \to u_0$ where $u_0$ is a minimizer in all compact subsets of $\R^n$. By considering properties of the free boundary of $u_0$ one may gather information on the free boundary of $u$ close to the point $x_0$. Theorem \ref{T: regularity} will guarantee that $u_0$ does exist. 

\begin{theorem}  \label{T: regularity}
Let $u$ be a minimizer in $B_1$. Then $u \in C^{0, s }(U)$ for all $U \Subset B_1$. 
\end{theorem}
For minimizers of \eqref{E: functional} we follow the method provided in \cite{lC10} for the one phase case. 

\begin{proof}
Throughout the beginning of the proof $C$ will be any constant depending on dimension $n$ and $a$. Let $u$ be a minimizer in $B_2$. For every $0<r<1$ we consider the harmonic replacement $v$ of $u$ in $B_r=B_r(x',0)$.  That is $\mathcal{L}_a v=0$ and $v=u$ on $\partial B_r$.  Since $u$ is a minimizer, $J(u) \leq J(v)$ in $B_r$, so
\[
\int_{B_r}{|x_n|^a |\nabla u|^2} \leq \int_{B_r}{|x_n|^a |\nabla v|^2} + C r^{n-1}
\] 
We now use that $\mathcal{L}_a v =0$, so 
\[
\int_{B_r}{|x_n|^a \langle \nabla v, \nabla (v-u)\rangle} = 0
\]
and this allows us to conclude
\[
\int_{B_r}{|x_n|^a |\nabla (u-v)|^2} \leq C r^{n-1}
\]
If we now choose $\rho < r < 1$
\begin{alignat*}{2}
\int_{B_{\rho}}{|x_n|^a |\nabla u|^2} &= \int_{B_{\rho}}{|x_n|^a |\nabla (u-v+v)|^2} \\
&\leq 2 \left(\int_{B_r}{|x_n|^a |\nabla (u-v)|^2} +  \int_{B_{\rho}}{|x_n|^a |\nabla v|^2}  \right) \\
&\leq C r^{n-1} + 2 \left(\frac{\rho}{r}\right)^{n-|a|} \int_{B_r}{|x_n|^a |\nabla v|^2} \text{ by Lemma \ref{L: menergy}} \\
&\leq C r^{n-1} + C \left(\frac{\rho}{r}\right)^{n-|a|} \int_{B_r}{|x_n|^a |\nabla u|^2}
\end{alignat*}
We now choose $\delta < 1/2$ with
\[
r = \delta^k , \qquad \rho = \delta^{k+1} , \qquad \mu \equiv \delta^{n-1}
\]
to obtain
\begin{equation}   \label{E: iteratebound}
\int_{B_{\delta^{k+1}}}{|x_n|^a |\nabla u|^2} \leq C \mu^{k} + C \mu \delta^{1-|a|} 
\int_{B_{\delta^k}}{|x_n|^a |\nabla u|^2}
\end{equation}
We now may choose $\delta$ such that $C \delta^{1-|a|} < 1$. Using a simple induction argument we conclude
\[
\int_{B_{\delta^k}}{|x_n|^a |\nabla u|^2} \leq \frac{C^2}{1-C\delta^{1-|a|}} \mu^{k-1}
\]
Then for all $r<1/2$ and a different constant which will also depend on the $L^2(a,B_2)$ norm of $\nabla u$
\begin{equation} \label{E: ebound1}
\int_{B_r(x',0)}{|x_n|^a |\nabla u|^2} \leq C r^{n-1}
\end{equation}
and so we may conclude as in \cite{lC10} that
\begin{equation}   \label{E: ebound2}
\int_{B_r(x',0)}{|\nabla u|} \leq Cr^{n-1 + s }
\end{equation}
Since the estimate \eqref{E: ebound2} is only true for balls centered on the thin space we cannot use Morrey's theorem to immediately conclude $C^{0,s}$ regularity for $u$ inside the solid ball $B_{1/2}$. However, one may use the proof of Morrey's theorem (as outlined in \cite{mZ97}) with the estimate \eqref{E: ebound2} to conclude 
\begin{equation}   \label{E: campanato}
|u(x',0)-\overline{u}_B| \leq Cr^{s}
\end{equation}
so that $u$ is $C^{0, s}$ on the thin space $\R^{n-1} \times \{0\}$. Equation \eqref{E: ebound2} and hence also \eqref{E: campanato} will hold for $|u|$. We now aim to conclude that we have the same H\"older growth off the thin space. By optimal H\"older regularity along the thin space, we only need to show H\"older growth in the pure $|x_n|$ direction. For a fixed point $(y',0)$, we consider the rescaled functions 
\[
u_r(x) \equiv \frac{u(y',0) + xr)-u(y',0)}{r^{s}}
\]
which have a universal (unweighted) $L^2$ gradient bound in $B^* = B_{1/2}(0, \ldots ,0,1)$ by \eqref{E: ebound1}. Using estimate \eqref{E: campanato} for $|u_r|$, we may deduce that the average value of $|u_r|$ over $B_{3/2}(0)$ is universally bounded; consequently, the average value of $|u_r|$ over $B^*$ will also be universally bounded. By using the (unweighted) Poincare inequality in $B^*$ we obtain
\[
\|u_r\|_{W^{1,2}(B_{1/2}(0, \ldots ,0,1))} \leq C 
\]
By first variation $\mathcal{L}_a u_r =0$ if $|x_n>0|$. By staying away from the thin space we may use regularity theory for uniformly elliptic equations and conclude that each $u_r$ is continuous in  $B^*$ and we have the weak Harnack inequality 
\[
\| u_r  \|_{L^{\infty}(B_{1/4}(0,\ldots ,0,1))}  \leq C
\] 
This proves the H\"older growth off the thin space. That is,
\begin{equation}  \label{E: hgrowth}
\frac{|u(y',0)-u(x)|}{|(y',0)-x|^{s}} \leq C
\end{equation}
Let now $x,y \in B_1$. If $|y_n| \leq |x-y|$ we may use \eqref{E: hgrowth} to bound
\[
\frac{|u(x)-u(y)|}{|x-y|^{s}}
\]
If $|y_n|>|x-y|$ then we may rescale with 
\[
u_r = \frac{u((x',0) + rx)-u(x',0)}{r^{s}}
\]
and use interior gradient bounds (in $B^*$ as defined before) on uniformly elliptic equations to  
conclude
\[
\frac{|u(x)-u(y)|}{|x-y|^{s}} \leq C
\]
\end{proof}

The H\"older regularity of minimizers allows us to conclude the following about the convergence of sequences of minimizers.
\begin{corollary} \label{C: holderconv}
\label{C: conv}
  Let $\{u_k\}$ be a sequence of minimizers of the functional \eqref{E: functional} in
  the domain $D$ with $\left\| u_k \right\|_{L^{\infty}(\partial D)}
  \leq M$. Then there exists a subsequence and a function $u_0$ such
  that for every open $U \Subset D$
  \begin{alignat*}{2}
    &(1)   &\quad & u_0 \in H^1(a,U) \cap C^{s}(\overline U) \\
    &(2)   && u_k \to u_0 \text{ in } C^{\beta}(\overline U) \text{ for } \beta < s\\
    &(3)   && u_k \rightharpoonup u_0 \text{ in } H^{1}(a,U)\\
  \end{alignat*}
\end{corollary}

\begin{proof}
Properties (1) and (2) follow immediately from the H\"older-regularity proven in Theorem \ref{T: regularity}. Property (3) follows from the inequalities \eqref{E: ebound1} and \eqref{E: asobolev}.
\end{proof}

Since minimizers are continuous, we may use the first-variation to conclude
\begin{proposition}   \label{P: solution}
Let $u$ be a minimizer of \eqref{E: functional} in $\Omega$. Then
\[
 \mathcal{L}_a u=0 \quad \text{in  } \Omega \setminus \Lambda(u)
\]
\end{proposition}

From Proposition \ref{P: solution} one expects the following
\begin{proposition}   \label{P: byparts}
Let $u$ be a minimizer in $\Omega$. For any ball $B \Subset \Omega$
\[
\int_{B}{|x_n|^a |\nabla u|^2} = \int_{\partial B}{|x_n|^a uu_{\nu}}
\]
\end{proposition}

\begin{remark}
Proposition \ref{P: byparts} holds for more general domains than a ball; however, the assumption that the domain is a ball will suffice for our purposes. 
\end{remark}

\begin{proof}
We define the following sequence of cutoff functions
\[
\eta_k(x) = 
   \begin{cases}
      0 , &\text{ if } d_x \leq 1/k \\
      kd_x - 1, &\text{ if } 1/k \leq d_x \leq 2/k  \\
      1  ,  &\text{ otherwise }
   \end{cases}
\]
Where $d_x = \text{dist}(x,\Lambda(u))$. Then $|\nabla \eta_k|=k$ when $1/k \leq d_x \leq 2/k$ and zero otherwise. We now use optimal regularity of $u$ to establish that the sequence $\eta_k u$ is bounded in $H^{1}(a,B)$. 
\begin{alignat*}{2}
\int_{B}{|x_n|^a |\nabla (\eta_k u)|^2} 
&\leq \int_{B}{2|x_n|^a \left( \eta_k^2|\nabla u|^2 + u^2 |\nabla \eta_k|^2 \right)} \\
&\leq \int_{B}{2|x_n|^a |\nabla u|^2} 
+ \int_{B \cap \{d_x \leq 2/k \}}{8|x_n|^a Ck^{a-1} k^2} \\
&\leq \int_{B}{2|x_n|^a |\nabla u|^2} 
+ \int_{B \cap \{|x_n| \leq 2/k \}}{8|x_n|^a Ck^{1+a} } \\
&\leq \int_{B}{2|x_n|^a |\nabla u|^2} + C \quad \text{ for some new constant } C
\end{alignat*}
Then there exists $v$ such that $\eta_k u \rightharpoonup v$ in $H^{1}(a,B)$ and $\eta_k u \to v$ pointwise by Proposition \ref{P: compact}. Since $\eta_k u \to u$ pointwise, then $v = u$. Now using the divergence theorem and that $\mathcal{L}_{a}u =0$ away from the coincidence set $\Lambda(u)$ we obtain
\[
\int_{B}{|x_n|^a \langle \nabla (\eta_k u) , \nabla u \rangle} = \int_{\partial B}{|x_n|^a \eta_k u u_\nu}
\] 
Then let $k \to \infty$ to obtain the result.   
\end{proof}

\section{Nondegeneracy and Weiss Monotonicity} \label{S: nondegeneracy}
When we have a blow-up sequence $u_r \to u_0$, it is not immediately obvious if $u_0$ could be degenerate, that is $u_0 \equiv 0$. If $u_0 \equiv 0$, then we would be unable to gather any information on the free boundary of $u$ near $x_0$. Theorem \ref{T: nondegeneracy} will guarantee that $u_0$ will not be degenerate. 

\begin{theorem}[Nondegeneracy] \label{T: nondegeneracy}
Fix $t > 0$, and let $u$ be a minimizer of $J$. There exists $\epsilon > 0$ with $\epsilon$ depending only on $\{\lambda^+, \lambda^-, t\}$ such that if $u|_{\partial B_r} \leq  \epsilon r^{s} $ $(u|_{\partial B_r} \geq  -\epsilon r^{s})$ then 
\[
u(x) \leq 0 \quad (u(x) \geq 0) \qquad \text{for } x \in B_{tr}'
\]
\end{theorem}

The proof of Theorem \ref{T: nondegeneracy} is included in the appendix and only requires slight modifications from the proof presented in \cite{mA11}.

\begin{corollary} \label{C: 1/2 growth}
If $u$ is a minimizer and $0 \in \Gamma^+$ $(0 \in \Gamma^-)$, then 
\begin{equation}   \label{E: 1/2 growth}
\sup_{\partial B_r}u \geq C r^{s} \qquad \left( \inf_{\partial B_r}u \leq -C r^{s} \right)
\end{equation}
Where $C$ depends only on $\lambda^+, \lambda^-$ and $n$.
\end{corollary}

In \cite{gW98} G. Weiss introduced a monotonicity formula for a free boundary problem that allowed one to conclude that blow-ups were homogeneous. Theorem \ref{T: weiss} gives a modified Weiss-type monotonicity formula that allows us to conclude Corollary \ref{C: monotone} namely, that all blow-ups are homogeneous of degree $s=(1-a)/2$. Corollary \ref{C: monotone} is crucial in proving Theorem I. 

\begin{theorem} \label{T: weiss}
Let $B_r = B_r(x_0,0)$. Define $W(r,u,x_0) = $
\begin{equation} \label{E: weiss}
\frac{1}{r^{n-1}} \left( \int_{B_r}{|x_n|^a|\nabla u|^2} 
+ \int_{B_r'}{\lambda^+ \chi_{ \{u>0\} } + \lambda^- \chi_{ \{u<0\} } } \right)
-\frac{s}{r^n} \int_{\partial B_r}{|x_n|^a u^2} 
\end{equation}
$W(r,u,x_0)$ is finite and monotone increasing in $r$. Furthermore, if $r_1 < r_2$, then  $W(r_1, u) = W(r_2, u)$ if and only if $u$ is homogeneous of degree $s=(1-a)/2$ on the ring $r_1 < |x| < r_2$. 
\end{theorem}

\begin{remark}
If $u_r(x) = \frac{u(rx)}{r^{s}}$, then  $W(r,u) = W(1,u_r)$.
\end{remark}

The proof of Theorem \ref{T: weiss} requires only slight modifications from the proof presented in \cite{mA11} for the case in which $a=0$ and therefore the proof is contained in the appendix.

\begin{corollary}  \label{C: monotone}
Let $u_r \to u_0$ a blow-up at $(x_0,0)$. Then $u_0$ is homogeneous of degree $s$
\end{corollary}

\begin{proof}
By \eqref{E: ebound1} and optimal regularity it is easy to verify that $W(0+,u,x_0)$ is bounded from below, so that 
\[
W(2r/3,u,x_0)-W(r/3,u,x_0) \to 0 \text{  as } r \to 0
\]
From the explicit representation of $W'$ provided in the proof of Theorem \ref{T: weiss} we may write
\begin{alignat*}{2}
W(2r/3)-W(r/3) 
	& = \int_{r/3}^{2r/3}{\frac{1}{\rho^{n-1}} 
	\int_{\partial B_{\rho}}{|x_n|^a \left(\frac{(1-a)u}{\sqrt{2}\rho} - \sqrt{2} u_{\nu} \right)^2} \ d\rho} \\
	& = \int_{1/3}^{2/3}{\frac{1}{(rt)^{n-1}} 
	\int_{\partial B_{rt}}{|x_n|^a \left(\frac{(1-a)u}{\sqrt{2}rt} - \sqrt{2} u_{\nu} \right)^2} r \ dt} \\
	& = \int_{1/3}^{2/3}{\frac{1}{t^{n-1}} 
	\int_{\partial B_{t}}{(rx_n)^a \left(\frac{(1-a)u(rx)}{\sqrt{2}rt} - \sqrt{2} \nabla u(rx)\cdot \nu  \right)^2} r \ dt} \\
	& = \int_{1/3}^{2/3}{\frac{1}{t^{n-1}} 
	\int_{\partial B_{t}}{|x_n|^a \left(\frac{(1-a)u_r}{\sqrt{2}t} - \sqrt{2} \nabla u_r\cdot \nu \right)^2}    \ dt}  \\
	& \geq   \int_{1/3}^{2/3}{ 
	\int_{\partial B_{t}}{|x_n|^a \left(\frac{(1-a)u_r}{\sqrt{2}t} - \sqrt{2} \nabla u_r\cdot \nu \right)^2}    \ dt}  \\
	& = \int_{B_{2/3} \setminus B_{1/3}}{ |x_n|^a\left(\frac{u_r}{\sqrt{2}t} - \sqrt{2} \nabla u_r\cdot \nu \right)^2} dx
\end{alignat*}
Now we use that $u_r \rightharpoonup u_0$ in $H^1(a, B_1)$ and $u_r \to u_0$ in $L^2(a,B_1)$ by Corollary \ref{C: holderconv}, so $u_0$ is homogeneous of degree $s$. 
\end{proof}

\section{Courant-Fischer Maximum-Minimum Principle}      \label{S: courant}
We may decompose the operator $\mathcal{L}_a$ into its radial and spherical parts similar to the case for the Laplacian. If $u \in H^1(a,S^{n-1})$, then $\mathcal{L}_a^{\theta} u = f$ is to be interpreted as
\[
-\int_{S^{n-1}}{|x_n|^a \langle \nabla_{\theta} u, \nabla_{\theta} v \rangle} = 
\int_{S^{n-1}}{|x_n|^a fv} 
\quad \text{for all } v \in H^1(a, S^{n-1}) 
\]
Corollary \ref{C: monotone} shows that all blow-ups are homogeneous. Homogeneous solutions of $\mathcal{L}_a u =0$ correspond to eigenfunctions on the sphere. Specifically, if $\mathcal{L}_a u=0$ and $u=r^{\alpha}f(\theta)$, then
\[
-\mathcal{L}_a^{\theta} f = \lambda f
\]
where $\lambda = \alpha(\alpha +n-2+a)$. We also have the converse.

\begin{lemma}  \label{L: spheresolution}
Suppose $-\mathcal{L}_{a}^{\theta} f = \lambda f$. If $u=r^{\alpha}f$ with $\lambda=\alpha(\alpha+n-2+a)>0$, $\alpha>0$, then $u$ is a weak solution to $\mathcal{L}_a u =0$ in $\R^n$.
\end{lemma}
\begin{proof}
Let $v \in H_{0}^1(a,B_R)$. Then
\begin{alignat*}{2}
\int_{B_R}{|x_n|^a \langle \nabla u , \nabla v \rangle} 
&=\int_{0}^{R} \int_{\partial B_{r}}{|x_n|^a \left( \frac{\langle \nabla_{\theta} u , \nabla_{\theta} v \rangle}{r^2} + u_{\nu} v_{\nu} \right) }  \\
&=\int_{0}^{R} \int_{\partial B_{r}}{|x_n|^a \left( r^{\alpha -2}\lambda fv + \alpha r^{\alpha -1} f v_{\nu} \right)}   \\
&=\int_{\partial B_1}{\cos^a(\theta_{n-1}) \alpha f} \left(\int_{0}^{R}{\frac{d}{dr}\left(r^{\alpha + n-2+a}v(r \theta) \right)dr}  \right) d\sigma
\end{alignat*}
Now
\[
\int_{0}^{R}{\frac{d}{dr}\left(r^{\alpha + n-2+a}v(r \theta) \right)dr} =0
\]
for a.e. $\theta$ since $v \in H_{0}^{1}(a, B_R)$, and thus the lemma is proven.
\end{proof}

To utilize the Courant-Fischer maximum-minimum principle we will need the following lemma.
\begin{lemma}
Let $\Omega \subset S^{n-1}$ be open. The spectrum of 
\[
-\mathcal{L}_{a}^{\theta} u: H^1(a, \Omega) \subset L^2(a,\Omega) \hookrightarrow L^2(a,\Omega)
\]
is a nonnegative sequence that is either finite or increases to infinity. 
\end{lemma}
\begin{proof}
From the Reisz representation theorem, for every $f \in L^2(a,\Omega)$ there exists a unique $u \in H^1(a, \Omega)$ such that for all $v \in H^1(a, \Omega)$ the following identity holds
\[  
\int_{\Omega}{|x_n|^a \left(\langle \nabla u, \nabla v \rangle + uv \right)} = \int_{\Omega}{|x_n|^a fv }
\]
We now aim to conclude that the operator $K: L^2(a,\Omega) \hookrightarrow L^2(a,\Omega)$ given by $K(f)=u$ is compact. 
To obtain a compactness theorem on $S^{n-1}$, for any $u \in H^1(a, \Omega)$ we extend $u$ radially by defining $\tilde{u} = \eta(r) u$ for $\eta$ a bump function on $\R$. Then for a sequence $u_k \in H^1(a, \Omega)$ we obtain a bounded sequence $\tilde{u}_{k} \in H_0^1(a, B_2 \setminus B_{1/2})$ and by Proposition \ref{P: compact} we obtain that for a subsequence $\tilde{u}_{k} \to \tilde{u} \in L^2(a,B_2 \setminus B_{1/2})$ and pointwise almost everywhere. Since $\tilde{u} = \eta(r) u$ for some $u \in H^1(a, \Omega)$ we conclude that $u_k \to u$ in $L^2(a,\Omega)$. We may therefore conclude that $K$ is compact. 

Now $-\mathcal{L}_a^{\theta} u = \lambda u$ in $\Omega$ if and only if $K(u)= \frac{1}{\lambda + 1} u$. From the theory of self-adjoint nonnegative compact operators we know that the spectrum of $K$ is either finite or a nonnegative sequence decreasing to zero. Then we obtain that the spectrum of $-\mathcal{L}_a^{\theta}$ is either finite or a nonnegative sequence increasing to infinity.
\end{proof}

If $\Omega = S^{n-1}$ then the first eigenvalue $\lambda_1 =0$. If $\Omega$ is a proper subset of $S^{n-1}$ such that $\Omega^c$ has positive capacity, then the first eigenvalue $\lambda_1 > 0$ and corresponds to the principle eigenfunction that is nonnegative.  
Let $\Omega \subset S^{n-1}$ be open and define $W=H_0^1(a,\Omega)$. 
To compare the eigenvalues of $V=H^1(a,S^{n-1})$ to those of the subspace $W=H_0^1(a,\Omega)$ we employ the Courant-Fischer maximum-minimum principle.
\begin{proposition}[Courant-Fischer]
Let $\Omega \subset S^{n-1}$ be open. The k-th eigenvalue of $-\mathcal{L}_a^{\theta} u$ associated to the domain $\Omega$ is determined by
\[
\lambda_k = \max_{S \in \Sigma_{k-1}} 
\min_{\overset{ v \in S^{\perp}}{\left\| v \right\|_{\mathcal{L}_a^2}=1}} 
\int_{S^{n-1}}{|x_n|^a|\nabla v|^2} 
\]
\end{proposition}
where $\Sigma_{k-1}$ is the collection of all $k-1$ dimesnional subspaces of $H_0^1(a, \Omega)$.
\begin{remark}
This principle is proven in \cite{MR0065391}.
\end{remark}

From this principle we conclude 
\begin{proposition}   \label{P: eigencomparison}
If $0 = \lambda_1 < \lambda_2 \leq  \ldots$ are the eigenvalues of $V$ and $\gamma_1 < \gamma_2 \leq \ldots$ are the eigenvalues of $W$, then 
\[
\lambda_k \leq \gamma_k \text{ for all } k
\] 
\end{proposition}

\begin{proof}
The proof is along the same lines of the proof of the maximum-minimum principle provided in \cite{MR0065391}. The only difference is we take a linear combination of the first $k$ eigenvectors in $W$ rather than in $V$.
Specifically, let $S$ be any $k-1$ dimensional subspace of $V$. Let $w_1, \ldots , w_k$ be the normalized eigenfunctions corresponding to the first $k$ eigenvalues of the subspace $W$. We may then construct
\[
w = \sum_{i=1}^{k}{c_i w_i} \quad \text{ with  } \sum_{i=1}^{k}{c_{i}^2} = 1
\]
and such that $w \in S^{\perp}$. Since the $w_i$ are orthogonal to each other we obtain that 
\[
\int_{S^{n-1}}{|x_n|^a |\nabla w|^2} = \sum_{i=1}^{k}{\gamma_i c_i^2 \leq \gamma_k}
\]
Thus we have shown that for $S$ any $k-1$ dimensional subspace of $V$
\[
\min_{\overset{ v \in S^{\perp}}{\left\| v \right\|_{\mathcal{L}_a^2}=1}} 
\int_{S^{n-1}}{|x_n|^a|\nabla v|^2} 
\leq \gamma_k
\]
Then by the Courant-Fischer maximum-minimum principle, $\lambda_k \leq \gamma_k$, and the proposition is proven. 
\end{proof}

\begin{theorem}    \label{T: originregularity}
Let $\mathcal{L}_a u = 0$ in all of $\R^n$ and let $u$ be homogeneous of degree $\alpha$ with $u(0)=0$. If $\alpha < \min \{1, 1-a\}$, then $u \equiv 0$. 
\end{theorem}

\begin{proof}
Solutions of $\mathcal{L}_a u =0$ are $C^1$ in any $(x',0)$ direction \cite{MR2367025}. Since $u$ is homogeneous of degree $\alpha < 1$, we may conclude that $u(x',0) \equiv 0$. (We must be differentiable in any $(x',0)$ direction at the origin.) We note that $x_n^{1-a}$ is the principle eigenfunction on $H_0^1(a,S_+^{n-1})$ since it is positive. Here $S_+^{n-1} = S^{n-1} \cap \{x_n > 0\}$. 
Now $u \in H_0^1(a,S_+^{n-1})$ and $\alpha < 1-a$, so the eigenvalue associated to $u$ is strictly less than that of the eigenvalue associated to that of $x_n^{1-a}$. Then $u \equiv 0$. 
\end{proof}

\begin{corollary}    \label{C: homogeneity}
Let $u$ be homogeneous of degree $\alpha$, having nontrivial positive and negative parts, continuous, and such that
\[
\mathcal{L}_a u(x) = 0
\]
whenever $x \notin \Lambda(u)$. Then $\alpha \geq \min \{1, 1-a\}$
\end{corollary}

\begin{proof}
Since $u$ is homogeneous, then $u$ is an eigenfunction of $\mathcal{L}_a^{\theta}$ on $\Omega = S^{n-1} \setminus \Lambda(u) $. If $\Omega$ is not connected, then $\Lambda(u)= B_{1}'$. Then by comparison with the principle eigenfunction $x_n^{1-a}$ (as in the proof of Theorem \ref{T: originregularity}) $\alpha \geq 1-a$. If $\Omega$ is connected, then since $u$ has nontrivial positive and negative parts, $u$ cannot be the principle eigenfunction. By Proposition \ref{P: eigencomparison} the eigenvalue $\gamma_2$ of $u$ is such that 
\[
\lambda_2 \leq \gamma_2
\]
where $\lambda_2$ is the eigenvalue corresponding to the first free eigenfunction $g$ on $S^{n-1}$. That is $\mathcal{L}_a^{\theta} g = \lambda_2 g$ on $S^{n-1}$. We may then define $v = r^{\beta}g$ where $\lambda_2 = \beta(\beta +n-2+a)$. Then $\mathcal{L}_a v = 0$ in $\R^n$ by Lemma \ref{L: spheresolution}, and so by Theorem \ref{T: originregularity} we know $\beta \geq \min\{1, 1-a \}$.  Since $\gamma_2 = \alpha (\alpha + n-2+a)$ and $\lambda_2 = \beta(\beta +n-2+a)$, we see then that $\alpha \geq \beta \geq \min \{1,1-a\}$. 
\end{proof}

\section{Separation of the Free Boundaries}      \label{S: separation}
We may now prove the main theorem of the paper. We first show the separation of the phases.  

\begin{theorem}   \label{T: separation}
Let $u$ be a minimizer. Then $\Gamma^+ \cap \Gamma^- = \emptyset$.
\end{theorem}

\begin{proof}
Suppose by way of contradiction that $x_0 \in \Gamma^+ \cap \Gamma^-$. Let $u_r \to u_0$ be a blow-up. By nondegeneracy (Corollary \ref{C: 1/2 growth}) and $C^{\beta}$ convergence (Corollary \ref{C: conv}) $u_0$ has nontrivial positive and negative parts. Also it follows from Corollary \ref{C: conv} that $\mathcal{L}_a u_0(x) =0$ if $x \notin \Lambda(u_0)$. By Corollary \ref{C: monotone} we know that $u_0$ is homogeneous of degree $s=(1-a)/2$. Since $(1-a)/2 < \min\{1, 1-a\}$, we obtain a contradiction to Corollary \ref{C: homogeneity}.
\end{proof}

We may now prove the second half of Theorem I. Namely, in a small neighborhood of each free boundary point a minimizer has a sign in the solid ball. 

\begin{theorem}    \label{T: solidsep}
Let $x_0 \in \Gamma^+ \ (x_0 \in \Gamma^-)$ then there exists $r>0$ depending on $x_0$ such that $u \geq 0 \ (u \leq 0)$ in the solid ball $B_r(x_0)$.
\end{theorem}

\begin{proof}
Without loss of generality we may assume $x_0 =0$. Let $u_{r_k} \to u_0$ be a blow-up of $u$ at the origin. Since $u_0$ is homogeneous of degree $s$,  Corollary \ref{C: homogeneity} allows us to conclude $u_0 \geq 0$ in all of $\R^n$. Since each $u_{r_k}(x',x_n)$ is
  $a$-harmonic in the open set $\{x \in B_{1/ r_k} \mid x_n \neq 0\}$,
  then $u_0$ will be $a$-harmonic in the open set $\{x \in \R^n \mid x_n
  \neq 0\}$. We define
  \[
  \delta = \inf u_0 \text{   over the set } B_1 \cap \{|x_n| \geq 1/2\}.
  \]
  We claim that $\delta>0$. Indeed, otherwise by the strong minimum
  principle (or Harnack inequality) $u_0\equiv 0$ in $\R^n_+$ or $\R^n_-$, and therefore
  $u_0\equiv 0$ on $\R^{n-1}\times\{0\}$. By nondegeneracy we know that on either $\R^n_+$ or $\R^n_-$ we have $u_0 >0$. Then by odd reflection we obtain a homogeneous (of degree $s=(1-a)/2$) function  $\tilde{u}_0$ that is $a$-harmonic in all of $\R^n$. This is a contradiction to Theorem \ref{T: originregularity}. So $\delta >0$.
  
   Then, by $C^{\alpha}$ convergence, for large
  enough $k$, $u_{r_k}(x',x_n) \geq \delta /2$ for $|x_n| \geq 1/2$ in
  $B_1$.  Also by $C^{\alpha}$ convergence, $\inf_{B_1} u_{r_k} \to
  0$. Now by thin separation, for large enough $k$,
  \[
  u_{r_k}(x',0) \geq 0 \text{ in } B_1'
  \]
  Without loss of generality it suffices to show that $u_{r_k} \geq 0$
  in $B_{1/2}^+$. Let $v_k$ be the $a$-harmonic function such that
  \[
  v_k|_{B_1'}=0 \text{, and } v_k |_{\partial B_1^+} = u_{r_k}
  \]
  Then $ v_k \leq u_{r_k} $ in all of $B_1^+$. We show for $k$ large
  enough that $v_k \geq 0$ in $B_{1/2}^+$.  To this end, consider two
  subsets $E_1$ and $E_2$ of $\partial({B_1^+})$:
$$
E_1=\partial(B_1^+)\cap \{x_n\geq 1/2\},\quad E_2=\partial(B_1^+)\cap
\{0<x_n<1/2\},\quad
$$
and there $a$-harmonic measures $\omega_1$ and $\omega_2$ with respect to
the domain $B_1^+$. The latter means that $\omega_i$ are $a$-harmonic
functions in $B_1^+$ satisfying
$$
\omega_i|_{\partial (B_1^+) }=\chi_{E_i},\quad i=1,2.
$$
By using  the
boundary Harnack inequality, one then has that
$$
c |x_n|^{1-a} \leq \omega_i(x)\leq C |x_n|^{1-a} \quad\text{in }B_{1/2}^+.
$$
for some positive constants $c$ and $C$ depending on $n$ and $a$.  Now, by
using the maximum principle we then can write that in $B_{1/2}^+$
\begin{align*}
  v_k(x)&\geq (\delta/2) \omega_1(x)+\omega_2(x)\inf_{B_1^+}v_k\\
  & \geq |x_n|^{1-a} [(\delta/2)c- C\sup_{(\partial B_1)^+} u_{r_k}^-].
\end{align*}
Since $u_{r_k}^- \to 0$ uniformly on compact subsets of $\R^n$, we
obtain that $v_k(x)\geq 0$ in $B_{1/2}^+$ for large $k$. This
completes the proof.
\end{proof}

\appendix
\section{Proof of Almgren's Formula}
\begin{proof}[Lemma \ref{L: almgren}]
The proof of Lemma \ref{L: almgren} relies on the following equality 
\begin{equation}    \label{E: almgrenlemma}
D'(r)= \frac{n-2+a}{r}D(r) + \int_{\partial B_r}{|x_n|^a 2u_{\nu}^2}
\end{equation}
\eqref{E: almgrenlemma} is \eqref{E: almgrenmin} in the case that $\lambda^+ = \lambda^- =0$ ($a$-harmonic functions are minimizers of \eqref{E: functional} when $\lambda^+ = \lambda^- =0$).  We also have 
\begin{equation}   \label{E: almgrenlemma2}
\int_{B_r}{|x_n|^a |\nabla u|^2} = \int_{\partial B_r}{|x_n|^a u u_{\nu}}
\end{equation}
We obtain \eqref{E: almgrenlemma2} by recalling that $\mathcal{L}_a u =0$ in $B_1$ and using $\eta_k u$ as a test function where $\eta_k$ is defined as in \eqref{E: etak}, then
\[
\int_{B_r}{|x_n|^a \langle \nabla u, \nabla (u \eta_k) \rangle} = 0
\]
By letting $k \to 0$ we obtain \eqref{E: almgrenlemma2}.
The monotonicity of $N(r)$ as well as case of equality then follow from \eqref{E: almgrenlemma} and \eqref{E: almgrenlemma2} exactly as shown in \cite{CS}. 
\end{proof}

\section{Proof of Nondegeneracy}   
We begin this section with the so called Lattice principle. Since minimizers are not necessarily unique, we may not necessarily conclude that if $u$ and $v$ are two minimizers with $u \leq v$ on $\partial D$, then $u \leq v$ in $D$. Instead we have the following theorem.

\begin{theorem}[Lattice Principle] \label{T: Maximumprinciple}
Let $u,v$ be two minimizers of the functional $J$ with $u|_{\partial D} \leq v$. If we define $w_1 \equiv \max \{u,v\}$ and $w_2 \equiv \min \{u,v\}$, then $w_1$ and $w_2$ are minimizers of the functional $J$.
\end{theorem}

\begin{proof}
It is fairly straightforward to check that 
\[
J(w_1) + J(w_2) = J(u) + J(v)
\]
Since $w_1 |_{\partial D} = v$ and $w_2 |_{\partial D} = u$, we conclude that $w_1$ and $w_2$ are minimizers of the functional $J$. 
\end{proof}

\begin{corollary} \label{C: nondegeneracy}
If the boundary data are symmetric about the line $(0,\dots ,0,x_n)$, then there is a maximal (minimal) minimizer, i.e. there exists a minimizer $u^*$ such that $v \leq u^*$ $(v \geq u^*)$ in $B$ for all other minimizers such that $v|_{\partial B} = u^*$. Furthermore, $u^*$ will be symmetric about the line $(0,\dots, 0 ,x_n)$
\end{corollary}

\begin{proof}
By Theorem \ref{T: Maximumprinciple} the maximum (minimum) of rotations will be a minimizer. $u^*$ may be obtained by a limiting procedure.
\end{proof}

To prove Theorem \ref{T: nondegeneracy} we will need the following two lemmas.

\begin{lemma}    \label{L: collapse}
There exists a modulus of continuity $\sigma $ with $\sigma (0)=0$ such that if $u_{\epsilon}$ is any minimizer such that $u |_{\partial B_1} \equiv \epsilon$, then  
\[
\int_{B_{1}'}{\lambda^+ \chi_{ \{u_{\epsilon} > 0 \} }} \leq \sigma (\epsilon) 
\]
\end{lemma}

\begin{proof}
Define
\[
v_{\epsilon} = 
\begin{cases}
  0 & \text{for } |x| \leq 1 - \sqrt{\epsilon} \\
	\sqrt{\epsilon}(|x|-1) + \epsilon & \text{ otherwise}
\end{cases}
\]
It is easy to see that $J(v_\epsilon) \to 0$ as $\epsilon \to 0$.
Now since 
\[
\int_{B_{1}'}{\lambda^+ \chi_{ \{u_{\epsilon} > 0 \} }} \leq J(u_{\epsilon}) \leq J(v_\epsilon)
\]
the lemma is proven. 
\end{proof}

This next lemma will strengthen Corollary \ref{C: nondegeneracy} in the case when our boundary values are identically constant.
\begin{lemma}     \label{L: steiner}
Let $u$ be a minimizer such that the values of $u|_{\partial B} = M$. Then $u$ is symmetric about the line $(0, \dots, 0, x_n)$, and the coincidence set $\Lambda(u) = \overline{B}_{\rho}'$ for some $\rho \geq 0$. 
\end{lemma}

\begin{proof}
Extend $u$ to be a function on the cube $Q$ with side length 2, by defining $u(x) = M$ for $x \notin B$. We now apply Steiner symmetrization (as defined in \cite[page 82]{bK85}) to the function $w = M - u$ on lines parallel to $\R^{n-1} \times \{0\} $.

If we only consider $\{x \mid \ |x_n| > \epsilon\}$,  then $w$ is Lipschitz. Then by \cite[page 82]{bK85}, if we Steiner symmetrize $w$ to obtain $v$ we get:
\[
\int_{B \cap \{|x_n|> \epsilon \} }{|x_n|^a|\nabla u|^2} = \int_{B \cap \{|x_n|> \epsilon \} }{|x_n|^a |\nabla w|^2} 
\geq \ \int_{B \cap \{|x_n|> \epsilon \} }{|x_n|^a|\nabla v|^2} 
\]
Equality is only achieved if $w$ (and hence $u$) is already Steiner symmetric along the lines we symmetrize.
Furthermore, $v$ will have the same boundary values as $w$ on $\partial B$. 
Then by letting $\epsilon \to 0$ we obtain
\[
\int_{B}{|x_n|^a|\nabla u|^2} = \int_{B}{|x_n|^a|\nabla w|^2} \geq \ \int_{B}{|x_n|^a|\nabla v|^2}
\]
Finally, we note that $\H^{n-1}(\{u=0\})$ is invariant under Steiner symmetrization. Then by a limiting process, we see that $u$ is a minimizer if and only if $u$ is symmetric about the line $(0, \dots ,0, x_n)$ and $\{u=0\}$ is a connected thin ball and centered at the origin.
\end{proof}

We are now able to prove the nondegeneracy result.
\begin{proof}[Theorem \ref{T: nondegeneracy}]
First we note that by rescaling we only need to prove Theorem \ref{T: nondegeneracy} on the unit ball $B$. Also, Theorem \ref{T: Maximumprinciple} and Corollary \ref{C: nondegeneracy} reduce Theorem \ref{T: nondegeneracy} to proving the theorem for the maximal minimizer $u_{\epsilon}^*$ where $u_{\epsilon}^*|_{\partial B} = \epsilon$. Lemma \ref{L: steiner} proves that 
\[
\Lambda(u_{\epsilon}^*)= \overline{B}_{\rho}'
\]
for some $\rho < 1$. Lemma \ref{L: collapse} shows
\[
\int_{B_1'} {\lambda^+ \chi_{ \{u_{\epsilon}^*>0\} }} \ \to 0 \text{ as } \epsilon \to 0
\]
Then there exists $\epsilon$ depending only on $\{t,\lambda^+ \}$ such that if $u |_{\partial B} = \epsilon$ then 
\[
u |_{ B_{t}'} \equiv 0
\] 
The case for which $u \geq -\epsilon$ is proven similarly.
\end{proof}

\section{Proof of Weiss monotonicity formula}

\begin{proof}[Theorem \ref{T: weiss}]
The proof is a slight modification of the proof for the case $a=0$ given in \cite{mA11}. Since $u$ is not necessarily differentiable we follow the ideas of using domain variation given by G. Weiss in \cite{gW98}.
Since our formula is defined for a ball centered on the $\R^{n-1} \times \{0\}$ plane, we may assume without loss of generality that $x_0 =0$. Let $\tau_{\epsilon}(x) = x + \epsilon \eta_{k} x$ where 
\begin{equation}    \label{E: etak}
\eta_{k}(x) = \max \left(0 , \min ( 1, \frac{r- |x|}{k} )  \right)
\end{equation}
Then $\eta_{k}(x) = 0$ outside of $B_r (0)$, and 
\[
\eta_{k}(x) \to \chi_{ \{B_r (0)\} } \text{ as } k \to 0
\]
Notice that $\tau_{\epsilon}(x) = x(1+ \epsilon \eta_k (x))$ leaves $\R^{n-1} \times \{0\}$ invariant. Now
\[
\nabla \eta_k (x) = \frac{-x}{|x|k} \chi_{ \{B_r \setminus B_{r-k} \} }
\]
and
\[
D \tau_{\epsilon} (x) = I + \epsilon \left(\eta_k (x) I + x \nabla \eta_k (x)  \right) + o(\epsilon)
\]
Now let $u_{\epsilon} \left( \tau_{ \epsilon} (x) \right) = u(x)$ and $y = \tau_{\epsilon}(x)$. Then
\[
\frac{1}{\epsilon} \left(J(u_{\epsilon}) - J(u)  \right) \geq 0
\]
and 
\begin{alignat*}{2}      
J(u_{\epsilon}) - J(u) 
& = \int_{D}{|y_n|^a|\nabla u_\epsilon (y)|^2} + \int_{D'}{\lambda^+ \chi_{\{u_\epsilon > 0 \} }
+ \lambda^- \chi_{\{u_\epsilon < 0 \}} }    \\
& \quad - \int_{D}{|x_n|^a|\nabla u (x)|^2} - \int_{D'}{\lambda^+ \chi_{\{u > 0 \} }
 + \lambda^- \chi_{\{u < 0 \}} }
\end{alignat*}
Now
\begin{alignat*}{2}
\text{det } D \tau_{\epsilon}(x)  &= 1 + \epsilon \ \text{ trace } D(\eta_k(x)x) + o(\epsilon) \\
\text{trace } D(\eta_k(x)x)  &= \text{ div } (\eta_k(x)x) \\
D \tau_{\epsilon}^{-1} &= I - \epsilon D(\eta_k(x)x) + o(\epsilon) 
\end{alignat*}
Then substituting these into the equality above we obtain that $J(u_\epsilon) -J(u)$
\begin{alignat*}{2}  
& = \int_{D}{|x_n + \epsilon \eta_k (x)x_n|^a|\nabla u(x)(D \tau_\epsilon (x))^{-1}|^2 \text{ det } D \tau_\epsilon} +o(\epsilon) \\
& \quad + \int_{D'}{\left(\lambda^+ \chi_{ \{u>0\} } + \lambda^- \chi_{ \{u<0\} } \right) \text{ det } D \tau_{\epsilon}' (x)} + o(\epsilon) \\
& \quad - \int_{D}{|x_n|^a | \nabla u|^2} - \int_{D'}{\lambda^+ \chi_{ \{u>0\} } + \lambda^- \chi_{ \{u<0\} }} \\
& = \int_{D}{|x_n + \epsilon \eta_k (x)x_n|^a\left(|\nabla u|^2 - 2 \epsilon \nabla u D(\eta_k (x)x) \nabla u \right) 
\left( 1+ \epsilon \text{ div }\eta_k (x)x  \right)} \\
& \quad + \int_{D'}{\left(\lambda^+ \chi_{ \{u>0\} } + \lambda^- \chi_{ \{u<0\} } \right) 
\left(1 + \epsilon \text{ div } \eta_{k}' (x',0)x'   \right)}  + o(\epsilon) \\
& \quad - \int_{D}{|x_n|^a| \nabla u|^2} - \int_{D'}{\lambda^+ \chi_{ \{u>0\} } + \lambda^- \chi_{ \{u<0\} }} \\
& = \int_{D}{|x_n + \epsilon \eta_k (x)x_n|^a|\nabla u|^2 - |x_n|^a|\nabla u|^2} \\
& \quad + \epsilon \int_{D}{ |x_n + \epsilon \eta_k (x)x_n|^a 
 \left( |\nabla u|^2 \text{ div } \eta_k (x)x -2 \nabla u D(\eta_k (x)x) \nabla u \right)} + o(\epsilon) \\
& \quad + \epsilon \int_{D'}{ \left(\lambda^+ \chi_{ \{u>0\} } + \lambda^- \chi_{ \{u<0\} } \right)
\left( \text{ div } \eta_{k}' (x',0)x'   \right)} + o(\epsilon)
\end{alignat*}

Now we may let $\epsilon$ be both positive and negative and the limit is the same, so
\[
\lim_{\epsilon \to 0} \frac{1}{\epsilon} \left[J(u_\epsilon) - J(u)  \right] = 0
\]
Then we obtain the following equality:
\begin{alignat*}{2}     
0 & = \int_{D}{a|x_n|^a |\nabla u|^2 \eta_k(x)} +
\int_{D}{|x_n|^a \left(|\nabla u|^2 \text{ div } \eta_k (x)x -2  \nabla u D(\eta_k (x)x) \nabla u \right)} \\
& \quad + \int_{D'}{\left(\lambda^+ \chi_{ \{u>0\} } + \lambda^- \chi_{ \{u<0\} } \right) 
\left( \text{ div } \eta_{k}' (x',0)x'   \right)}
\end{alignat*}
We have 
\begin{alignat*}{2}
\text{div } \eta_k (x)x = n \eta_k (x) - \frac{|x|}{k}\chi_{B_r \setminus B_{r-k}} \\
\text{div }(\eta_k (x',0)x') = 
(n-1) \eta_{k}' - \frac{|x'|}{k} \chi_{B_{r}' \setminus B_{r-k}'}
\end{alignat*}
Then
\begin{alignat*}{2}
0 &= (n-2+a)\int_{B_r}{|x_n|^a |\nabla u|^2 \eta_{k}} - \frac{1}{k}\int_{B_r \setminus B_{r-k}}{|x| |x_n|^a \left( |\nabla u|^2 - 2|\langle \nabla u, \frac{x}{|x|} \rangle|^2 \right)}   \\
  & + \quad (n-1)\int_{B_{r}'}{(\lambda^+ \chi_{ \{u>0\} } + \lambda^- \chi_{ \{u<0\} }     )\eta_{k}'}  \\
  & \quad - \frac{1}{k}\int_{B_{r}' \setminus B_{r-k}'}{|x'|(\lambda^+ \chi_{ \{u>0\} } +    \lambda^- \chi_{ \{u<0\} } )}
\end{alignat*}
so as $k \to 0$
\begin{equation}   \label{E: almgrenmin}
\begin{aligned}
0 & = (n-2+a)\int_{B_r}{|x_n|^a |\nabla u|^2} - r \int_{\partial B_r}{|x_n|^a \left( |\nabla u|^2 -2u_{\nu}^2 \right)} \\
  & \quad + (n-1)\int_{B_{r}'}{\lambda^+ \chi_{ \{u>0\} } + \lambda^- \chi_{ \{u<0\} }}
    -r \int_{\partial B_{r}'}{\lambda^+ \chi_{ \{u>0\} } + \lambda^- \chi_{ \{u<0\} }} \\
  & = (n-1) \int_{B_r}{|x_n|^a|\nabla u|^2} - r \int_{\partial B_r}{|x_n|^a|\nabla u|^2} \\
  & \quad + (n-1) \int_{B_{r}'}{\lambda^+ \chi_{ \{u>0\} } + \lambda^- \chi_{ \{u<0\} }}
    -r \int_{\partial B_{r}'}{\lambda^+ \chi_{ \{u>0\} } + \lambda^- \chi_{ \{u<0\} }} \\
  & \quad -(1-a) \int_{B_r}{|x_n|^a|\nabla u|^2} + 2r \int_{\partial B_r}{|x_n|^a u_{\nu}^2} 
\end{aligned}
\end{equation}
By Proposition \ref{P: byparts} 
\[
\int_{B_r}{|x_n|^a |\nabla u|^2} = \int_{\partial B_r}{|x_n|^a u u_{\nu}} 
\]
so
\begin{alignat*}{2}0 & = (n-1) \int_{B_r}{|x_n|^a |\nabla u|^2} - r \int_{\partial B_r}{|x_n|^a|\nabla u|^2} \\
  & \quad + (n-1) \int_{B_{r}'}{\lambda^+ \chi_{ \{u>0\} } + \lambda^- \chi_{ \{u<0\} }}
    -r \int_{\partial B_{r}'}{\lambda^+ \chi_{ \{u>0\} } + \lambda^- \chi_{ \{u<0\} }} \\
  & \quad - (1-a)\int_{\partial B_r}{|x_n|^a u \cdot u_{\nu}} + 2r \int_{\partial B_r}{|x_n|^a u_{\nu}^2} 
\end{alignat*}
Now multiply both sides of the equation by $-r^{-n}$ to obtain that for almost every $r$
\begin{alignat*}{2}
0 & = \left[\frac{1}{r^{n-1}} \int_{B_r}{|x_n|^a|\nabla u|^2} \right]'
    + \left[\frac{1}{r^{n-1}} \int_{B_r '}{\lambda^+ \chi_{ \{u>0\} } + \lambda^- \chi_{ \{u<0\} }} \right]' \\
  & \quad  
   - \frac{1}{r^{n-1}} \int_{\partial B_r}{|x_n|^a\left(\frac{(1-a)u u_\nu}{r} - 2 u_{\nu}^2 \right)}
\end{alignat*}
For $\epsilon < r$ we may integrate and use Fubini's theorem to obtain
\begin{alignat*}{2}
\int_{\epsilon}^{r} \frac{1}{\rho^{n-1+a}} \int_{\partial B_{\rho}}{|x_n|^a 2uu_{\nu}} d\sigma \ d\rho
&=\int_{\partial B_1} \int_{\epsilon}^{r}{|x_n|^a 2u(\rho x)u_{\nu}(\rho x)} d\rho d\sigma  \\
&=\int_{\partial B_1} |x_n|^a \int_{\epsilon}^{r}{ 2u(\rho x)u_{\nu}(\rho x)} d\rho d\sigma  \\
&=\int_{\partial B_1}{ |x_n|^a  \left(u^2(rx) - u^2(\epsilon x) \right) d \sigma } \\
&=- c + \frac{1}{r^{n-1+a}} \int_{\partial B_r}{|x_n|^a u^2} d \sigma 
\end{alignat*}
So for almost every $r$
\begin{equation}      \label{E: der}
\frac{\text{d}}{\text{dr}}\left[\frac{1-a}{2r^n} \int_{\partial B_r}{|x_n|^a u^2} \right] = 
\frac{1}{r^{n-1}} \int_{\partial B_r}{|x_n|^a \left( \frac{(1-a)u u_{\nu}}{r} - \frac{(1-a)^2u^2}{2r^2}  \right)}
\end{equation}
We then add and subtract the piece from \eqref{E: der} to obtain for almost every $r$
\begin{alignat*}{2}
0 & = \left[\frac{1}{r^{n-1}} \int_{B_r}{|x_n|^a|\nabla u|^2} \right]'
    + \left[\frac{1}{r^{n-1}} \int_{B_r '}{\lambda^+ \chi_{ \{u>0\} } + \lambda^- \chi_{ \{u<0\} }} \right]' \\
  & \quad - \left[\frac{1-a}{2r^n}  \int_{\partial B_r}{|x_n|^a u^2} \right]' 
    - \frac{1}{r^{n-1}} \int_{\partial B_r}{|x_n|^a \left(\frac{(1-a)u}{\sqrt{2}r} - \sqrt{2} u_{\nu} \right)^2}
\end{alignat*}

Thus, $W' \geq 0$, and $W' = 0$ on the interval $r_1 <r< r_2$ if and only if $u$ is homogeneous of degree $s=(1-a)/2$ on the ring $r_1 < |x| < r_2$.
\end{proof}

\bibliographystyle{amsplain}
\bibliography{Bibliography}

\end{document}